\documentclass[12pt,reqno]{amsart}

\usepackage{graphicx} 
\usepackage{amsmath, amsthm,mathabx, amssymb,nicefrac,indentfirst}

\usepackage{amsmath, amsthm, amsopn, amssymb, enumerate}

\usepackage{amsmath , amssymb, latexsym }
\usepackage{graphics}
\usepackage{graphicx}
\usepackage{verbatim} 

\usepackage[mathscr]{eucal}

\theoremstyle{plain}
\newtheorem{theorem}{Theorem}[section]
\newtheorem{corollary}[theorem]{Corollary}
\newtheorem*{corollary*}{Corollary}
\newtheorem{proposition}[theorem]{Proposition}
\newtheorem{lemma}[theorem]{Lemma}
\newtheorem*{proposition*}{Proposition}
\newtheorem*{theorem*}{Theorem}
\newtheorem*{lemma*}{Lemma}

\newtheorem{claim}{Claim}
\newtheorem*{claim*}{Claim}

\theoremstyle{definition}
\newtheorem{definition}{Definition}[section]
\newtheorem*{definition*}{Definition}

\theoremstyle{remark}

\newtheorem*{obs*}{Observation}

\newtheorem{example}{Example}

\newcommand{\Z}[1]{\mathbb{Z}_{#1}}
\newcommand{\prob}[1]{\mathbb{P}\left(#1\right)}
 
\newcommand{\case}[1]{\textbf{\emph{Case #1}}:}

\def\ZZ{\mathbb{Z}}
\def\NN{\mathbb{N}}

\newcommand{\Bin}{\ensuremath{\mathrm{Bin}}}
\newcommand{\sign}{\ensuremath{\mathrm{sign}}}
\newcommand{\rt}{\right}
\newcommand{\lt}{\left}

\newcommand{\cG}{\ensuremath{\mathcal{G}}}

\title{Virtually Fibering Random Right-Angled Coxeter Groups}

\date{\today}

\author{Gonzalo Fiz Pontiveros}
\address{} 
\email{}
\author{Roman Glebov}
\author{Ilan Karpas}

\begin{document}
\maketitle

\begin{abstract}
 
 

We show that the Right-Angled Coxeter group $C=C(G)$ associated to a random graph $G\sim \mathcal{G}(n,p)$
with $\frac{\log n + \log\log n + \omega(1)}{n} \leq p < 1- \omega(n^{-2})$ virtually algebraically fibers.
This means that $C$ has a finite index subgroup $C'$ and a finitely generated normal subgroup $N\subset C'$
such that $C'/N \cong \mathbb{Z}$.
We also obtain the corresponding hitting time statements, more precisely, we show that as soon as $G$ has minimum degree at least 2 and as long as it is not the complete graph, then $C(G)$ virtually algebraically fibers.
The result builds upon the work of Jankiewicz, Norin, and Wise and it is essentially best possible.
\end{abstract}

\section{Introduction}

A group $K$ \emph{virtually algebraically fibers} if there is a finite index subgroup $K'$ 
admitting a surjective homomorphism $K'\to \ZZ$ with finitely generated kernel. 
This notion arises from topology: a $3$-manifold $M$ is virtually a surface bundle over a circle 
precisely when the fundamental group of $M$  virtually algebraically fibers (see the result of Stallings~\cite{Sta61}).



	
A \emph{Right-Angled Coxeter group} (RACG) $K$ is a group given by a presentation of the form 
\[\left\langle x_1, x_2, \ldots x_n \;|\;x_i^2, [x_i, x_j]^{\sigma_{ij}}\;: 1\leq i< j\leq n\right\rangle  \]
where $\sigma_{ij}\in \{0,1\}$ for each $1\leq i <j\leq n$. One can encode this information with a graph $\Gamma_{K}$ whose vertices are the generators $x_1,\ldots, x_n$ and $x_i\sim x_j$ if and only if $\sigma_{ij}=1$. Conversely given a graph $G$ on $n$ vertices, we will denote the corresponding RACG by $K(G)$.

Random Coxeter groups have been of heightened recent interest, see
for instance Charney and Farber~\cite{charney2012random}, Davis and Kahle~\cite{davis2014random}, and Behrstock, Falgas-Ravry, Hagen, and Susse~\cite{behrstock2015global}.

 Recently, Jankiewicz, Norin, and Wise~\cite{Jankiewicz_Virtually} developed a framework to show virtual fibering of  a RACG using Betsvina-Brady Morse theory~\cite{Bestvina_Morse_1997} and ultimately translated the virtual fibering problem for $K$ into a combinatorial game on the graph $\Gamma_K$. The method was successful on many special cases and also allowed them to construct examples where Betsvina-Brady cannot be applied to find a virtual algbraic fibering. 
 
  A natural question to consider is whether this approach is successful for a `generic' RACG, i.e., given a probability measure $\mu_n$ on the set of RACG's of rank at most $n$, is it true that a.a.s. as $n\to \infty$, a group sampled from $\mu_n$ virtually algebraically fibers. This question is also considered in \cite{Jankiewicz_Virtually}, specifically they consider sampling $\Gamma_K$ from the Erd\H{o}s-Renyi random graph model $\mathcal{G}(n,p)$ and they prove the following result:
  \begin{theorem}[Jankiewicz-Norin-Wise]
\label{JNW}
	Assume that  \[\frac{(2\log{n})^{\frac{1}{2}}+\omega(n)}{n^{\frac{1}{2}}}\leq p < 1 -\omega(n^{-2}),\]
	and let $G$ be sampled from $\mathcal{G}(n,p)$. Then, asymptotically almost surely, the associated Right-Angled Coxeter group $K(G)$ virtually algebraically fibers.
	\end{theorem}
	
  In this paper we extend this result to the smallest possible range of $p$, in fact we prove a hitting time type result. Namely we show that as soon as $\Gamma_K$ has minimum degree $2$ then a.a.s. $K$ virtually algebraically fibers.
    
  \begin{theorem}
  	\label{Main}
  	Let $G_0,G_1,\ldots, G_{\binom{n}{2}}$ 
  	denote the random graph graph process on $n$ vertices 
  	where $G_{i+1}= G_i\cup \{e_i\}$ 
  	and $e_i$ is picked uniformly at random from the non-edges of $G_i$.
  	Let $T=\min_{t}\;\{t\;: \delta(G_t)=2\}$, 
  	then a.a.s. the random graph process is such that 
  	$K(G_m)$ virtually algebraically fibers if and only if 
  	$T\leq m <\binom{n}{2}$. 
  	In particular for any $p$ satisfying
  	\[\frac{\log{n}+\log\log{n}+\omega(n)}{n}\leq p < 1-\omega(n^{-2})\]
  	and $G~\mathcal{G}(n,p)$,
  	the random Right-Angled Coxeter group $K(G)$ virtually algebraically fibers a.a.s.
  	 
  \end{theorem} 

The paper is structured as follows. 
In Section~\ref{sec:legalsystems}, we establish the graph-theoretic framework used in the remainder of the paper, and show that the minimum degree condition is in fact necessary for $n\geq 3$ and  hence Theorem~\ref{Main} is best possible.

In Section~\ref{sec:dense},
we look at the opposite extreme and prove Theorem~\ref{Main}
for very large $p$.
The proof presented in Section~\ref{WeakBound} mainly serves to provide
the reader with the concepts and the intuition used later;
it shows Theorem~\ref{Main} for most of the range of the edge probability.
In Section~\ref{sec:construction}, we present the construction used for the final part of the proof of Theorem~\ref{Main}.
Then in Section~\ref{sec:pseudorandom} we prove Theorem~\ref{Main} in the remaining case in the pseudorandom setting, i.e., we prove the statement for every graph satisfying certain (deterministic) properties.
Finally, in Section~\ref{sec:proof} we put the pieces together, and show that indeed in the remaining interval for $p$ in Theorem~\ref{Main}, the random graph a.a.s. satisfies the conditions required in Section~\ref{sec:pseudorandom}, thus completing the proof.

??
\subsection{Notation}
$V$ always denotes the vertex set;
floor/ceiling;
$G(n,p)$ and relation to the random graph process;
$\log$ is base $e$

\section{Legal Systems}
\label{sec:legalsystems}

In this section we follow the definitions in \cite{Jankiewicz_Virtually} to present the combinatorial game introduced in \cite{Jankiewicz_Virtually} used to construct virtual algebraic fiberings of Right-Angled Coxeter groups.

\begin{definition}
	Let $G=(V,E)$ be a graph. We say that a subset  $S\subset V$ is a \emph{legal state} if both $S$ and $V\setminus S$ are non-empty {\em connected subsets} of $V$, i.e., the corresponding induced graphs are connected and non-empty.
	\end{definition}

\begin{definition}
	For each $v \in V$, a \emph{move at $v$} is a set $M_v\subseteq V$ satisfying the following:
	\begin{itemize}
		\item $v\in M_v$
		\item $N(v)\cap M_v=\emptyset$
	\end{itemize}
	Let $\mathcal{M}=\{M_v \; : v \in V\}$ denote a set of moves.
	\end{definition}
We will identify subsets of $V$ as elements of $\Z{2}^{V}$ in the obvious way. Thus each state and each move correspond to  elements of  $\Z{2}^{V}$ and we will think of moves acting on states via group multiplication (or addition in this case).

\begin{definition}
    For a graph $G$, a state $S\subseteq V(G)$, and a set of moves $\mathcal{M}=\{M_v \; : v \in V\}$, the triple $(G, S, \mathcal{M})$ is a \emph{legal system} if for any element $g \in \langle\mathcal{M}\rangle$, $g(S)$ is a legal state of $G$.
\end{definition}

\begin{theorem}[\cite{Jankiewicz_Virtually}]
Let $(G,S,\mathcal{M})$ be a legal system, then the RACG $K(G)$ must virtually algebraically fiber.
\end{theorem}

To elucidate the notion of a legal system, let us  look at some toy examples (see Figure~\ref{fig:examples}) and ask whether each of these graphs contains a legal system.
\begin{figure}
\label{fig:examples}
  \caption{A couple of toy examples.}
  \centering
  \vspace{0.5cm}
    \includegraphics[width=0.5\textwidth]{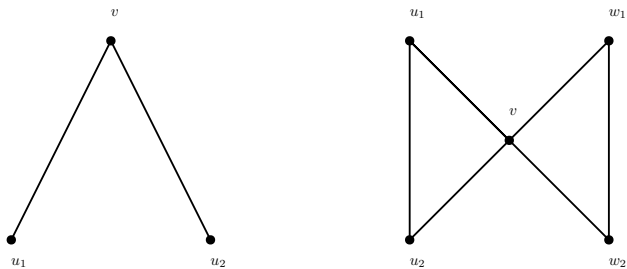}
\end{figure}

\begin{example} \label{cherry}
Let $G=(V,E)$ be a graph with three vertices $V=\{v,u_1,u_2\}$ and two edges $E=\{\{v,u_1\},\{v,u_2\}\}$. We show that $G$ has a legal system. Our initial legal state will be $S=\{u_1\}$. For our set of moves we choose $M_v=\{v\}$ (note that this is the only possible choice for the move at $v$), $M_{u_1}=M_{u_2}=\{u_1,u_2\}$. Then the group generated by the moves of the graph, written as a collection of sets, is $\langle\mathcal{M}\rangle=\{\{v\},\{u_1,u_2\},\{v,u_1,u_2\},\emptyset\}$. Hence, for any element $g\in \langle\mathcal{M}\rangle$, $g(S)$ is either a set of the form $\{u_i\}$ or $\{v,u_i\}$, for $i=1,2$, and in any case a legal state. Thus, $(G,S,\mathcal{M})$ is a legal system.
\end{example}

The graph in Example \ref{cherry} is unique in the sense that it is the only graph with a vertex of degree $1$ on at least $3$ vertices which contains a legal system. We prove this later in Proposition \ref{degree 1}.

Next, we look at an example of a graph without a legal system. We proceed by exhaustion.

\begin{example} \label{bowtie}
Let $V=\{v,u_1,u_2,w_1,w_2\}$, $E=\{\{v,u_i\},\{v,w_i\}, \{w_1,w_2\},\{u_1,u_2\}\}$, $i=1,2$. Let $G=(V,E)$. Assume by contradiction that $(G,S,\mathcal{M})$ is a legal system. Since $v$ is connected to all other vertices in the graph, we must have $M_v=\{v\}$. For the same reason, $v$ can not belong to any other move apart from $M_v$. Hence, we can assume without loss of generality that $v \notin S$.  Since $S$ is a connected subset of $V$, we can again assume without loss of generality that $S=\{u_1\}$ or $S=\{u_1,u_2\}$.

In the latter case, $M_{w_i}=\{u_1,u_2,w_i\}$ for $i=1,2$, because by the definition of a move, it must be the case that $\{w_i\}\subseteq M_{w_i} \subseteq \{w_i,u_1,u_2\}$, and if $u_1$ or $u_2$ would not belong to $M_{w_i}$, then $M_{w_i}S$ would not be a legal state. But then the set $\{w_1,w_2\} \in \langle\mathcal{M}\rangle$, and $\{w_1,w_2\}S=\{w_1,w_2,u_1,u_2\}$ is not a legal state. In the former case, from similar consideration, it must be the case that $M_{w_i}=\{w_i,u_1\}$ for $i=1,2$, but then again $\{w_1,w_2\}\in \langle\mathcal{M}\rangle$, and $\{w_1,w_2\}S=\{w_1,w_2,u_1\}$ is not a legal state. 
\end{example}

Next we show that Theorem~\ref{Main} is is essentially best possible. In fact, any graph on more than $3$ vertices with minimum degree at most $1$ does not have a legal system. 

\begin{proposition} \label{degree 1}
Let $G$ be a graph on $n$ vertices with $n\geq 4$ and suppose that $\delta(G)\leq 1$. Then $G$ does not have a legal system.   
\end{proposition}
\begin{proof}
For graphs with isolated vertices the statement is obvious, therefore we can assume that $\delta(G)=1$.
	We argue by contradiction. Suppose there exists an $S\subset V(G)$ and a set of moves $\mathcal{M}$ such that the triple $(G,S, \mathcal{M})$ is a legal system. Let $v$ be a vertex with $d(v)=1$ in $G$ and let $u$ be its unique neighbour. Since $u\notin M_v$ and $v\notin M_u$ we may assume without loss of generality that both $u,v \in S$ (if not then simply take a suitable translate). Observe that $v\in M_u(S)$ and $u\notin M_u(S)$. Furthermore, by our assumption the set $M_u(S)$ is connected and thus $M_u(S)=\{v\}$. Recall that $M_u$ is a set of non-neighbours of $u$ together with $u$ itself, and hence $S=M_u(M_u(S))=\{u,v\}$ which in turn implies that $M_u=\{u\}$. 
	
	\begin{claim*}
For every $g \in \langle \mathcal{M} \rangle$, we have that either 
	\begin{equation}
	\label{or}
	g(S)\in \{\{v\},\{u,v\}\} \text{\; or \; } M_v(g(S))\in  \{\{v\},\{u,v\}\}.	
	\end{equation}
		
	\end{claim*}

	Note that $u$ either belongs to both sets $g(S)$ and $M_v(g(S))$ or to neither of them, since $u \notin M_v$, whereas $v$ belongs to exactly one of these sets. Assume without loss of generality that $v \in g(S)$. If $u \in g(S)$, then $M_u(g(S))=g(S) \setminus \{u\}$ is a connected set containing $v$ but not $u$, and thus must be equal to $\{v\}$. This means that $g(S)=\{u,v\}$, providing \eqref{or}.
	
	If, on the other hand, $u \notin g(S)$, then $g(S)$ is a connected set which contains $v$ but not $u$, which again means that $g(S)=\{v\}$, again providing \eqref{or}.
	
	Thus, at least half the sets in $\left\{ g(S):~g\in \langle \mathcal{M} \rangle\right\}$  
	are either $\{ v \}$ or $\{ u,v \}$, which means that $\left|\left\{ g(S):~g\in \langle \mathcal{M} \rangle\right\}\right| \leq 4$, and therefore $M_w\in \left\{M_v, M_v\cup\{u\}\right\}$ for any $w\neq u$. Hence, $w \in M_v$ for any $w \neq u$, which means that $M_v=V\setminus\{u\}$. Furthermore, as $G$ has no isolated vertices we must have that $M_w=V\setminus\{u\}$ for any $w\neq u$ an hence $G$ must be in fact a star. The only way $M_v(M_u(S))=V\setminus \{u,v\}$ can be connected is if $n\leq 3$, a contradiction.
		\end{proof}

\section{Very dense regime}
\label{sec:dense}
In this section we show Theorem~\ref{Main} in the simpler range of very dense graphs.
\begin{theorem}
\label{densethm}
Let $G \in \mathcal{G}(n,m)$, i.e., a graph with $m$ edges picked uniformly at random. Suppose that $0.98\binom{n}{2}\leq m< \binom{n}{2}$. Then a.a.s $G$ has a legal system.       	
\end{theorem}	

\begin{proof}
Let $H$ denote the complement of $G$ and observe that $H\sim \mathcal{G}(n,t)$ where $t={n\choose 2}-m$. 
The strategy to find a legal system is a simple one: first we find a maximal matching $F=\{\{u_1,v_1\},\ldots ,\{u_k,v_k\}\} \subset H$. Then let $S=\{u_i\;: 1\leq i \leq k\}$ and for each $1\leq i \leq k$ set $M_{u_i}=M_{v_i}=\{u_i,v_i\}$ and $M_v=\{v\}$ for all $v\notin F$. We claim that with high probability this defines a legal system for $G$. 

Note that $V \in \langle \mathcal{M} \rangle$ and hence for any $g \in  \langle \mathcal{M} \rangle$, the complement of $g(S)$ can be expressed as $V \setminus g(S)=(Vg)(S)$, in other words the orbit of $S$ is closed under taking complements. In particular, to prove the claim, it is enough to show that for any $g \in \langle \mathcal{M} \rangle$, the set $g(S)$ is connected.

Furthermore, since $H$ contains at least one edge, $F$ must also be non-empty and so $g(S)\neq \emptyset$ for any $g\in \langle\mathcal{M}\rangle$. Thus it is sufficient to show that for every $g\in \langle\mathcal{M}\rangle$, the set $g(S)$ is connected.

By maximality of $F$, we know that $G[V\setminus F]$ is a clique in $G$ (equivalently an independent set in $H$). Hence, by our choice of moves, the only way that $g(S)$ can fail to be connected is if there exists some $v\in V$ such that 
\begin{equation}\tag{$\star$}
\label{cond}	
|N_{H}(v)\cap \{u_i,v_i\}|\geq 1 \text{ for at least $\lceil k/2\rceil$ indices } i\in [k]. 
\end{equation}

We now consider two cases.\\
\case{1} $t=o(n^{\frac{1}{2}})$. Observe that the expected number of paths of length two in $G(n,t)$ is at most $n^3(\frac{2t}{n^2})^2 \to 0$. In particular, by Markov, with high probability no two edges are incident in $H$. In particular \eqref{cond} cannot happen with high probability.\\
\case{2} $t=\Omega(n^{\frac{1}{2}})$. Observe that the expected number of independent sets of size $l$  in $\mathcal{G}(n,t)$ is

	$$O\left(n^{l}\left(1-\frac{2t}{n^2}\right)^{l^2/2}\right) =O\left( n^{l}e^{-\frac{1}{2}n^{-\frac{3}{2}}l^2}\right).$$

In particular, with high probability $H$ has no independent set of size $\Omega(n^{\frac{3}{4}})$. It follows that with high probability  $|F|=(1-o(1))n=2k$. On the other hand, if \eqref{cond} occurs, we must have that there exists $v \in V$ such that $d_H(v)\geq k/2$, and by Chernoff ??add reference to chernoff from somewhere, perhaps?? the probability of such high degree vertex is vanishingly small.  

\end{proof}
 \begin{corollary}
 	Let $G\in \mathcal{G}(n,p)$ where $0.99\leq p< 1- \omega(n^{-2})$. Then a.a.s. $G$ has a legal system.
 \end{corollary}
\begin{proof}
Sampling from $G$ from $\mathcal{G}(n,p)$ is equivalent to first choosing a random number $m\sim \text{Bin}({n\choose 2},p)$ of edges and then sampling $G$ from $\mathcal{G}(n,m)$. For $p$ in the above range we have that a.a.s. $0.98{n\choose 2}\leq m < {n\choose 2}$ and the corollary follows follows from Theorem \ref{densethm}.
\end{proof}

Observe that this upper bound is also optimal since for $p=1-cn^{-2}$, the probability that $G$ is in fact the complete graph is bounded away from $0$ and it is easy to see that the complete graph cannot have a legal system.

\section{A weaker bound}
\label{WeakBound}

Before we attempt to prove the main result of the paper we will give here a simple proof for a slightly smaller range of $p$. Namely we will show the following:

\begin{theorem} \label{big-p}
Let $\frac{3\log n}{n}\leq p \leq 0.99$. Then a.a.s. $G\sim \mathcal{G}(n,p)$ has a legal system. 
\end{theorem}

This achieves several purposes. We will be able to already introduce some of the ideas and statements required for the following section, motivate definitions in the construction and also present simplified computations by having a more restricted range of $p$. 
	

An \emph{equitable} colouring of a graph $G$ 
is a proper colouring of the vertices of $G$, where the sizes of any two colour classes differ by at most $1$. The {\em equitable chromatic number} of $G$ is the smallest integer $k$ such that there exists an equitable colouring of $G$ with $k$ colours.

We use the following theorem of Krivelevich and Patk\'{o}s~\cite{KrivPat09}.

\begin{theorem}[Krivelevich-Patk\'{o}s~\cite{KrivPat09}]
\label{K-P}
Let $G\sim \mathcal{G}(n,p)$. There exists a constant $C$ such that asymptotically almost surely the following holds:

\begin{itemize}
        \item[(a)] If \; $ \frac{C}{n}\leq p\leq \log{n}^{-8}$, then  \[\chi_{=}(G)\leq \frac{np}{(1-o(1)\log{(np)}}.\]
    \item[(b)] If\; $ \log{n}^{-8}<p<0.99$, then 
    \[\frac{n}{2\log_b{n}-\log{\log_b{n}}}\leq \chi_{=}(G)\leq \frac{n}{2\log_b{n}-8\log{\log_b{(np)}}},\]
    where $b=\frac{1}{1-p}$.
\end{itemize}

Note that when $p \to 0$, then $\log_b{n}-\log{\log_b{(np)}} \sim \frac{\log{(np)}-\log\log{(np)}}{p}$.
\end{theorem}  

We are now ready to prove Theorem~\ref{big-p}.

\begin{proof}[Proof of Theorem~\ref{big-p}]

By Theorem \ref{K-P} we know that a.a.s. we can find an equitable colouring of $G$ with $m=\Theta\left(\frac{np}{\log{(np)}}\right)$ colours.
Call the colour classes $C_1,\ldots, C_m$ and set $M_v=C_i$, where $C_i$ is the colour class that $v$ belongs to. So $v \in M_v$ and $N(v) \cap M_v=\emptyset$, as required. Let $S$ be a random subset of $V$ where each $v \in V$ is included into $S$ independently with probability $\frac{1}{2}$. 

Note that, as in the proof of \ref{densethm}, $V \in \langle \mathcal{M} \rangle$ and hence it is enough to show that for any $g \in \langle \mathcal{M} \rangle$, the set $g(S)$ is connected and non-empty.

The following well known lemma essentially reduces  the task to proving that none of these sets contains an isolated vertex.

\begin{lemma}
\label{Csubsets}
	Let $G \in \mathcal{G}(n,p)$ and $S\subset V(G)$ with $|S|\geq c n$ for some $c>0$. Then
\[\prob{S\text{ is not connected}}= O\left(\prob{S \text{ contains an isolated vertex}}\right)=O\left(ne^{-c np}\right).\] 
\end{lemma}

Notice that for every colour class $C_i$ and every state $g(S)$, the intersection $g(S)\cap C_i$ is either equal to $S\cap C_i$ or its complement $C_i \setminus S$.
By well-known estimates on large deviation in binomial distribution, we observe that a.a.s. it is true that for almost every colour class $C_i$, we have $|S\cap C_i|\sim |C_i|/2$.
Therefore, a.a.s. it is true that $|g(S)|> 2n/5$ for every state $g\in \langle\mathcal{M}\rangle$.
Furthermore, the orbit of $S$ is of size $2^m$, where all moves only depend on the chosen equitable colouring of $G$ and not on $S$. The crucial observation here is that for any $g\in \langle\mathcal{M}\rangle$, the distribution of $g(S)$ is the same as that of $S$. Thus, by the union bound and  Lemma~\ref{Csubsets}, the probability that the triple $(G,S,\mathcal{M})$ is not a legal system is at most
\[\sum_{g \in\langle\mathcal{M}\rangle} \prob{g(S) \text{ is not connected}}\leq o(1)+ \exp\left(\frac{np}{\log{(np)}}-\frac{2}{5}np+\log{n} \right)=o(1).\] 	
\end{proof}

\section{Construction}
\label{sec:construction}

The aim of this section is to outline our recipe to construct a legal system for $G \sim \mathcal{G}(n,p)$. The core idea behind the construction is the same as in \S \ref{WeakBound}. 
Ideally, we could simply choose a random initial set $S$, where each vertex in $G$ is added to the set with probability $\frac{1}{2}$. Then, the move at each vertex $v$ would be the colour class of vertex $v$ for an equitable colouring $C_1,\dots,C_m$ with $O(\log{n}/\log{\log{n}})$ colours, which we know exists w.h.p. from Theorem \ref{K-P}. This is the approach taken in the proof of Theorem \ref{big-p}, but it does not work for all $p$ in the range of Theorem \ref{Main}.

The main obstruction in this range are vertices with only few neighbours in either $S\cap C_i$ or $C_i\setminus S$ for many of the colour classes $C_i$. This could happen for the obvious reason that a vertex simply has very few neighbours in $G$, or it is an unlikely (and unlucky, for that particular vertex) choice of the random set $S$. The idea is to show that one may deterministically modify our initial random set $S$ to take care of the problematic vertices. It is in this  sets of vertices and their neighbourhoods that the modifications take place. The construction is as follows.

\begin{itemize}
    \label{sketch}
    \item Let $D_0$ denote vertices of degree at most $\frac{\log{n}}{100}$. Assign two unique neighbours to each vertex of $D_0$. Call the set of such neighbours $N_0$, and set $V'=V(G) \setminus D_0$.
    \item Partition $V'$ into large almost equitable independent sets $C_1, C_2, \ldots, C_m$ with $m\sim  \frac{np}{{\log{np}}}\leq(1+o(1))\frac{\log n}{\log{\log{n}}} $. We can do this by first partitioning $V$ into equitable colour classes and then taking away vertices in $D_0\cup N_0$ as the size of this set will be negligible compared to the size of the colour classes.
    
    \item Assign $+$ and $-$ signs to vertices of $G$ independently at random with probability $\frac{1}{2}$ and let $C_i^{+}=\{ v \in C_i:~ \sign(v)= +\}$ and $C_i^{-}=\{ v \in C_i :~ \sign(v)= -\}$.
    \item Define the function $\kappa: 2^{V}\to \NN$ as 
    \[\kappa(U)= \min_{\sigma \in \{+,-\}^{m}}\sum_{i} \left|U\cap C_i^{\sigma(i)}\right|,\]
    and set $D_1=\left\{ v \in V' : \kappa(N(v))< \log{\log{n}}^2\right\}$. As before we assign a pair of unique neighbours to each vertex in $D_1$ with the further property that they both lie on the same colour class and not in $D_0\cup N_0$. We call this set of neighbours $N_1$.

    \item
    We reassign to these pairs of vertices in $N_0$ and $N_1$ signs $+$ and $-$, so that for each pair one vertex is assigned $+$, and the other with $-$.
    Set $V''= V'\setminus D_1$.
    \item For every vertex $v \in V''$, set $M_v= C_i$ for the unique $i$ such that $v\in C_i$, for every $v \in D_0\cup D_1$ set $M_v=\{v\}$ and for for every $v \in N_0$ set $M_v={u,v}$ where $u$ is the unique vertex in $N_0$ such that $N(u)\cap N(v)\neq \emptyset$. Furthermore, we set our initial activated set to be $S=\{v \in V: \sign(v)=+\}$.
    \end{itemize}
\begin{figure}[h]
  
  \centering
  \vspace{0.5cm}
    \includegraphics[width=0.4\textwidth]{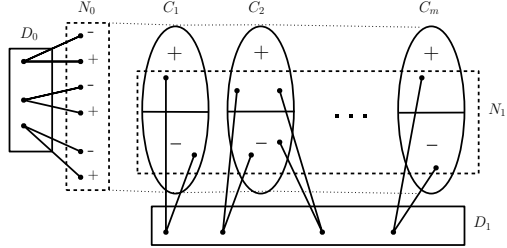}
    \caption{Picture of construction}
\end{figure}

\section{Proof of Main Theorem}
We will tackle the proof of the main theorem as follows: first we will give a small list of deterministic properties of a graph (which we call pseudorandom properties) that are sufficient to guarantee that the construction in the previous section indeed yields a legal system a.a.s. Finally we will complete the proof by showing that a random graph (at the appropriate density) a.a.s. presents all of the required pseudorandom properties. A caveat: there are two independent probability spaces at play in our approach: one is given by the random graph, and the other by the random 2-colouring in the construction. The first a.a.s. statement above is with respect to the latter space and the second with respect to the former.

\subsection{Sparse pseudorandom graphs}
\label{sec:pseudorandom}

\begin{theorem}
\label{pseudorandom}
For sufficiently large integer $n$, define $t=2n \log \log n / \log n$ and let $G$ be an $n$-vertex graph with $D_0:=\lt\{v\in V:~d(v)\leq \log n /100\rt\}$ satisfying the following:
\begin{enumerate}[(i)]
\item
\label{mindeg}
$\delta(G)\geq 2$,
\item
\label{maxdeg}
$\Delta(G) =O(\log n)$,

\item
\label{sized0}
$|D_0|\leq n^{0.9}$,

\item
\label{noshortpath}
there exists no non-trivial path of length at most $4$ with both endpoints in $D_0$,
\item
\label{chromatic}
$m:=\chi_{=}(G) =O(\log n / \log \log n)$,
\item
\label{mindeglarge}
every set $A\subseteq V(G)$ satisfying $\delta(G[A])> \log \log n^2/2$ is of size at least $t$,
\item
\label{twolargesets}
between any two disjoint sets $A,B\subseteq V(G)$ of sizes at least $t$, there exists an edge in $G$ between $A$ and $B$.
\item
\label{k23}
$G$ is $K_{2,3}$-free.
\end{enumerate}
Then $G$ has a legal system. 
\end{theorem}

Again, we start by assigning either $+$ or $-$ to every vertex of $G$ uniformly at random.
As mentioned earlier, the subtle point where the proof of Theorem~\ref{big-p} cannot be applied here, are the few vertices that behave irregularly. Following the description in the sketch above, let us choose two neighbours $v_+, v_-$ for every vertex $v\in D_0$ such that no vertex is chosen twice - this is possible because of Properties~\eqref{mindeg} and~\eqref{noshortpath}. Denote the set of all such chosen neighbours by $N_0$, and reassign the signs of vertices in $N_0$ according to their subscripts. 

Furthermore, let us fix an arbitrary equitable colouring of $G$ with $m$ colours, and denote the colour classes by $D_1, \ldots, D_m$, set $C_i=D_i\setminus D_0$ and observe that $|C_i|=(1-o(1))|D_i|$ by property (\ref{sized0}). As described in the sketch, for every $i\in [m]$ we define $C_i^+$ and $C_i^-$ to be the set of all vertices in $C_i$ with the corresponding sign. We would like to have a function that counts the minimum number of neighbours of any vertex in a set that contains either $C_i^+$ or $C_i^-$ for every $i\in [m]$. Towards that aim, we define $V'=V\setminus D_0$, $\kappa: 2^{V}\to \NN$ as 
    \[\kappa(U)= \min_{\sigma \in \{+,-\}^{m}}\sum_{i} \left|U\cap C_i^{\sigma(i)}\right|,\]
    and set $D_1=\left\{ v \in V' : \kappa(N(v))< \log{\log{n}}^2\right\}$.

In order to work with the exceptional vertices in $D_1$, we need the following lemma, analogous to Property
~\eqref{noshortpath} for $D_0$. We remark here that the set $D_1$ is a \emph{random subset} of $V'$ as it depends on the intial choice of 2-colouring.

\begin{lemma}
\label{noshortpath1}
A.a.s. for every vertex $v\in V$, there are at most $1000$ paths of length $2$ between $v$ and vertices in $D_1$.
\end{lemma}

Before we prove Lemma~\ref{noshortpath1}, we need to make the following technical statements.

\begin{claim}
\label{domin}
Let $Y\sim \Bin(m, \frac{1}{2})$ where $m\geq 1$ and $X=\min\{Y, m-Y\}$ then $X$ dominates $Z$ where $Z\sim\Bin\left(\left\lfloor\frac{m}{2}\right\rfloor, \frac{1}{2}\right)$, that is for all $t\geq 0$ we have that 
\[\prob{X\leq t}\leq \prob{Z\leq t}.\]
\end{claim}

\begin{proof}
We argue by induction. For $m=1,2$ the claim obviously holds. Assuming that it holds for $m=m_0$, we show that it also holds for $m=m_0+2$. Observe that $Y\sim Y'+W$ where $Y'\sim \Bin(m_0,\frac{1}{2})$ and $W\sim \Bin(2,\frac{1}{2})$ are independent. Furthermore, observe that $X$ dominates $X_1+X_2$ where $X_1=\min\{Y',m_0-Y'\}$ and $X_2=\min\{W,2-W\}$. By the induction hypothesis, letting $Z_1\sim \Bin\left(\left\lfloor\frac{m}{2}\right\rfloor,\frac{1}{2}\right)$ and $Z_2\sim \Bin(1,\frac{1}{2})$ be independent random variables, we know that $X_i$ dominates $Z_i$ for $i=1,2$. Using the independence of $X_1$ and $X_2$ and of $Z_1$ and $Z_2$ it follows that $X_1+X_2$ dominates $Z_1+Z_2$ and hence $X$ also dominates $Z_1+Z_2\sim \Bin\left(\left\lfloor\frac{m+2}{2}\right\rfloor\right)$ as claimed.
\end{proof}

\begin{claim}
\label{coupling}
Let $Y_1, \ldots, Y_k$ be independent random variables with $Y_i \sim \Bin\left(m_i, \frac{1}{2}\right)$ and $m_i\geq 1$ for every $i\in [k]$. Denote 
$X=\sum_{i=1}^{k}\min\{Y_i, m_i-Y_i\}$.
Then $X$ dominates $Z\sim \Bin\left(\sum_i\left\lfloor\frac{m_i}{2}\right\rfloor , \frac{1}{2}\right)$.
\end{claim}
\begin{proof}
Let $X_i=\min\{Y_i, m_i-Y_i\}$, then the $X_i$'s are independent random variables and $X=\sum_i X_i$. By Claim~\ref{domin}, we know that there exist independent random variables $Z_i\sim \Bin\lt(\lt\lfloor\frac{m_i}{2}\rt\rfloor, \frac{1}{2}\rt)$ such that each $X_i$ dominates $Z_i$ respectively. By independence of the $Z_i$'s, we then have that $X$ dominates $\sum_i Z_i\sim Z$.
\end{proof}
We can finally turn back our attention to the random set $D_1$.
\begin{lemma}
\label{sized1}
For any $U\subset V'$, let $X=X(U)=\kappa(N(u))$. Suppose that $|U|\geq \frac{\log n}{101}$, then, $$\prob{X(U)\leq 2 \log \log n^2}\leq n^{-1/300}.$$ 
\end{lemma}

\begin{proof}
By Claim~\ref{coupling}, we see that $X$ dominates $Y\sim \Bin \lt(|N(u)|/2, \frac{1}{2}\rt)$ and thus
\[\prob{X< 2\log\log n^2} \leq 
\prob{Y< 2\log\log n^2}
\leq 
\prob
{\Bin \lt( \log n/202, \frac{1}{2} \rt) < 2\log\log n^2}
<n^{-1/300}.\]
\end{proof}

We are now ready to prove Lemma~\ref{noshortpath1}.

\begin{proof}[Proof of Lemma~\ref{noshortpath1}]
By Property~\eqref{maxdeg}, every vertex $v\in V$ has $O(\log^2 n)$ vertices that are at distance at most $2$ from $v$. 
Therefore, if the statement of the lemma was to be wrong, by Property~\eqref{k23} there would be such $v$ where at least $1000$ of the $O(\log^2 n)$ vertices at distance at most $2$ from $v$ would be all in $D_1$. 
Although the events $X'(u_i)=\{u_i\in D_1 \}$ are not mutually independent, they are {\em almost} independent. 
Namely, for an arbitrary collection of $1000$ vertices $u_1, \ldots, u_{1000}$, 
the events 
\[X''(u_i) :=\mbox{``}\kappa\lt(U_i \rt) < 2 \log \log n^2 \mbox{''},\] where $U_i=N(u_i)\setminus \bigcup_{j\neq i}N(u_j)$ are mutually independent since $U_i\cap U_j=\emptyset$ for $i\neq j$. Furthermore  $X'(u_i) \implies X''(u_i)$ for every $i$. Finally, by Property~\eqref{k23} we see that
\[\lt|N(u_i)\setminus \bigcup_{j\neq i}N(u_j) \rt|> |N(u_i)|-2000 \geq \frac{\log n}{101},\]
and obtain
\[\prob{ \bigwedge_{i\leq 1000}X'(u_i)}
\leq 
\prob{\bigwedge_{i\leq 1000}X''(u_i)}
=
\prod_{i\leq 1000}\prob{ X''(u_i) }
<
\prod_{i\leq 1000}\prob{X(U_i)< 2\log\log n^2}
<n^{-1.1}.\]

The lemma now follows by a union bound over all choices for $v\in V$ and all choices of $1000$ vertices $u_i$ at distance at most $2$ from $v$.
\end{proof}

    As before we assign a pair of unique neighbours to each vertex in $D_1$ with the further property that they both lie on the same colour class.
    This is possible since by Lemma~\ref{noshortpath1}, for every $v\in V$ at most $1000$ vertices from $D_1$ have joint neighbours with $v$, and by Property~\eqref{k23} every such vertex has at most $2$ joint neighbours with $v$, whereas $v$ has a total of at least $\log n /100$ neighbours in $V$, out of which at most one is in $D_0\cup N_0$ by Property~\eqref{noshortpath}.
    As with the vertices in $N_0$, we
    assign to these two vertices signs $+$ and $-$, and set $V''= V'\setminus D_1$. 
    
    As described in the sketch, 
    for every vertex $v\in V''$, 
    we set $M_v= C_i$ 
    for the unique $i$ such that $v\in C_i$ for every $v \in D_0\cup D_1$ set $M_v=\{v\}$ and for for every $v \in N_0$ set $M_v={u,v}$ where $u$ is the unique vertex in $N_0$ such that $N(u)\cap N(v)\neq \emptyset$.

    To finish the proof, all that is left is to prove the following claim:

\begin{claim}
Let $S=\{v\in V: \sign(v)=+\}$, the triple $(G,S, \mathcal{M})$ is a legal system.
\end{claim}
\begin{proof}
As in the proof of Theorem ~\ref{big-p}, $V \in \langle \mathcal{M} \rangle$, thus to prove the claim it is enough to prove that $g(S)$ is connected for every $g \in \langle \mathcal{M} \rangle$.
Observe that, by construction, for any $g\in \mathcal{M}$ and any vertex $v\in D_0\cup D_1$, out of the two vertices $v_+, v_- \in N_0\cup N_1$ exactly one is in $g(S)$.
Therefore, for every $g \in \langle\mathcal{M}\rangle$, no vertex from $D_0\cup D_1$ is isolated in $g(S)$.
Notice
\[g(S)\cap V''=\bigcup_{i=1}^{m}C_i^{\sigma(i)}\cap V''\]
for some $\sigma\in \{+,-\}^m$. 

Suppose for the sake of contradiction that there exists such a vector $\sigma\in \{+,-\}^m$ for which the set $X=\bigcup_{i=1}^{m}C_i^{\sigma(i)}\cap V''$ is not connected. Then there must exist a subset $A\subset X$ such that $e(A,X\setminus A)=0$. 

Consider an arbitrary vertex $v\in A$. 
Since $v\in V''$, we have
\begin{equation*}
|N(v)\cap (g(S)\cup N_0\cup N_1)|\geq \log\log n^2.
\end{equation*}
Furthermore, by Property~\eqref{noshortpath}, $\lt|N(v)\cap \lt(D_0\cup N_0\rt)\rt|\leq 1$, and by Lemma~\ref{noshortpath1}, 
\newline $\lt|N(v)\cap \lt(D_1\cup N_1\rt)\rt|\leq 4000$.
Therefore, $|N(v)\cap A|>\log \log n^2 /2$, or in other words \newline $\delta(G[A])>\log \log n^2 /2$.
By Property~\eqref{mindeglarge} this implies that $|A|\geq t$. 
Analogously, $|X\setminus A|\geq t$, and Property~\eqref{twolargesets} guarantees the existence of an edge between $A$ and $X\setminus A$, a contradiction.
\end{proof}

\subsection{Putting the pieces together}
\label{sec:proof}

Theorems~\ref{big-p} and~\ref{JNW} show that $G\sim \cG(n,p)$ a.a.s. has a legal system for $p\geq 3 \log n/n$. 
Furthermore, for $p\leq \log n/n$, a.a.s. $G$ has a vertex of degree at most $1$, and by Proposition~\ref{degree 1} it does not have a legal system for $n\geq 4$.
Therefore, it suffices to show that in the range 
$\log n/n<p< 3 \log n/n$, the graph $G$ a.a.s. satisfies Properties~\eqref{maxdeg}--\eqref{k23} from Theorem~\ref{pseudorandom}.

Properties~\eqref{maxdeg} and~\eqref{k23} are well-known to hold a.a.s. in this range of $p$. 
Furthermore, Property~\eqref{chromatic} holds  a.a.s. as an immediate consequence of Theorem~\ref{K-P}. Property ~\eqref{noshortpath} also holds a.a.s. Indeed, this is just a special case of Claim~4.4 in~\cite{BSKrSu11} and  Theorem ~4.2.9 in \cite{thesis}.

We show the remaining Properties ~\eqref{mindeglarge}  and ~\eqref{twolargesets} in two separate lemmas.

\begin{lemma}
\label{lem:mindeglarge}
Let $G\sim \cG(n,p)$ with $\frac{\log n}{n}<p< \frac{3\log n}{n}$. 
Then a.a.s. every $A\subseteq V(G)$ satisfying 
$\delta(G[A])>\log \log n^2$ is of size at least $2n\log\log n /\log n$.
\end{lemma}

\begin{proof}
By Chernoff's inequality, 
the probability that a set $A$ of size $a\leq 2n\log\log n /\log n$ induces more than $a\log\log n^2/5$ edges, 
is $\exp\lt[-\Omega(a\log\log n^2)\rt]$.
Applying the union bound over all such sets provides the statement of the lemma.
\end{proof}

\begin{lemma}
\label{lem:twolargesets}
Let $G \sim \cG(n,p)$ with $\frac{\log n}{n}<p$ and let $t=2n \log \log n / \log n$. Then for any two disjoint sets $A, B \subseteq V(G)$ of sizes at least $t$, there exists an edge in $G$ between $A$ and $B$.
\end{lemma}

\begin{proof}
Observe that it is enough to prove the theorem for any two sets of size exactly $t$ (assume for simplicity $t$ is an integer). Call a pair of disjoint sets  $A, B \subseteq V$  of size $t$ \textit{bad}, if there is no edge between $A$ and $B$. The probability that such a given pair $A$ and $B$ is bad, is at most

\begin{equation}
\label{eq:badpair}
    (1-p)^{t^2}<e^{-pt^2}<e^{-2t\log \log n }.
\end{equation}.

The number of disjoint pairs of sets of size $t$ $A, B \subseteq V$
is at most

\begin{equation}
\label{eq:numofpairs}
{n \choose t}^2 <(en/t)^{2t}<(2e^{\log \log n -\log \log \log n})^{2t},
\end{equation}

so by ~\eqref{eq:badpair}, ~\eqref{eq:numofpairs} and the union bound, the probability that a bad pair in $G$ exists, is at most 
$e^{-t \log \log \log n}=o(1)$.

\end{proof}

\nocite{*}
\bibliography{legal.bbl}
\bibliographystyle{plain}
\end{document}